\documentclass[reqno,12pt]{amsart}
\usepackage[dvipdfmx]{graphicx,color}
\usepackage{mathrsfs}
\newtheorem{theorem}{Theorem}[section]
\newtheorem{pro}[theorem]{Proposition}
\newtheorem{cor}[theorem]{Corollary}
\newtheorem{lemma}[theorem]{Lemma}

\theoremstyle{definition}

\theoremstyle{remark}

\numberwithin{equation}{section}

\begin{document}

\title[self-reciprocal polynomials and real binary forms ]{ Discriminants of a class of self-inversive polynomials and real binary forms  }

\author{ Keisuke Uchimura}
\address{Department of Mathematics, Tokai University,  Hiratsuka, 259-1292, Japan}
\email{uchimura@tokai-u.jp. }



\begin{abstract}\hspace{0cm}

A  class of self-inversive polynomials includes all the  self-reciprocal polynomials.  
Let A denote the set of all self-reciprocal polynomials with  n+1 coefficients.  Let B   denote the set of certain  self-inversive and non self-reciprocal  polynomials  with  n+1 coefficients for odd  n.  Let C denote the set of real binary n-ic forms.  Then there exist  a bijection between A and C and  another  bijection between B and C.  Let f be a monic polynomial in A and g be the corresponding  polynomial in C. 
If the reading coefficient of g is not zero, then the discriminant of g is expressed by the determinant of a matrix of type (n, n).  Any element of the matrix is a polynomial in the coefficients of f with integer coefficients.  The same  holds for monic polynomials in B. 
\end{abstract}

\maketitle
\section{Introduction}
\subsection{Self-inversive polynomials}\hspace{0cm}
 
 A polynomial  
\[P(z) = a_0z^n + a_{1}z^{n-1} +  \cdot \cdot \cdot   + a_n\] is said to be a {\it self-inversive polynomial} of degree \(n\) if it satisfies \(a_0  \ne 0\) and \(P(z) = \mu P^*(z)\) where \(\mid \mu \mid = 1\) 
and 
\[P^*(z) = z^n \overline{P(1/\bar z)} = \bar a_nz^n + \bar a_{n-1}z^{n-1} + \cdot \cdot \cdot   + \bar a_n,\]
or equivalently,  
\begin{equation}  
a_0 \ne 0 \enskip \mbox{ and}  \enskip a_k = \mu\bar a_{n-k}   \enskip \mbox{ for every}  
\enskip 0 \le k \le n.  
\end{equation} 

See e.g. Marden \cite{M}.

In particular, if  \(P(z) = P^*(z), \quad P(z)\) is called be  {\it self-reciprocal}.  In the literature \cite{Sh} such a polynomial is called a  self-inversive polynomial.

The zeros of  a  self-inversive polynomial are on or symmetric in the unit circle \(C : \mid z \mid = 1\).   

In 1922, Cohn \cite{Co} proved that a polynomial \(P(z)\) has  all of its zeros on \(C\) if and only if it is self-inversive and its derivative \(P'(z)\) has all its zeros in the closed unit disk \(\mid z \mid \le 1\). 

Chen \cite{Ch} found another necessary and sufficient condition for which all zeros of  a self-inversive polynomial lie on \(C\).

\subsection{Real binary forms in terms of the complex variable}\hspace{0cm} 

Zeeman \cite{Z} described a real binary cubic form
\[g_2(x,y) = ax^3 + bx^2y + cxy^2 + dy^3, \quad  (a, b, c, d  \in {\mathbb R})\]
in terms of the complex variable \(z = x + iy\)  and showed that \(g_2\) can be expressed uniquely as 
\[g_2(x,y) = Re(\alpha z^3 + \beta z^2\bar z) , \quad  \alpha, \beta,   \in {\mathbb C}.\]

Poston and Stewart \cite{PoS} applied this method to the quartic forms.  A real binary quartic form
\[g_3(x,y) = ax^4 + bx^3y + cx^3y^2 + dxy^3 + ey^4\]
can be expressed uniquely as
\[g_3(x,y) = Re(\alpha z^4 + \beta z^3\bar z + \gamma z^2\bar z^2), \quad  \alpha, \beta   \in {\mathbb C},\quad \gamma  \in {\mathbb R}.\]

In this paper we will study relation between real binary forms and certain self-inversive polynomials via the terms of the complex variable.  Such self-inversive polynomials appear in several branches of mathematics.
\subsection{Monic self-inversive polynomials}\hspace{0cm}

Gongopadhyay and Parker \cite{GP} and Gongopadhyay, Parker and Parsad \cite{GPP}  classified  the dynamical action in SU(p,q)  using the coefficients of their characteristic polynomial.  In the case \(p+q = 4\),  the characteristic polynomial is 
\[\chi(X) = X^4 - \tau X^3 + \sigma X^2 - \bar\tau X +1, \quad \tau \in {\mathbb C},\quad \sigma  \in {\mathbb R}.\]
The locus where the resultant  \(R(\chi, \chi') = 0\)  was studied by Poston and Stewart \cite{PoS}.  The locus was named the holy grail.  The characteristic polynomial  \(\chi(X)\) is a monic self-reciprocal polynomial.

Uchimura \cite{U2}  studied the dynamics of holomorphic endomorphisms \(P_{A_3}^d\)  on   ${\mathbb P}^3({\mathbb C})$ which is related to a complex Lie algebra of type \(A_3\).  The set of the critical values of \(P_{A_3}^d\)  restricted to a real three-dimensional subspace of    ${\mathbb C}^3$ was proved to be equal to the holy grail.   Uchimura \cite{U2}  studied the relation between real binary quartic forms  and monic self-reciprocal polynomials of degree four.  In \cite{U1},  monic self-inversive polynomials of degree three were studied.

Marden \cite{M} stated that in a linear difference equation with constant coefficients, the requirement for a stable solution is that all the zeros of the characteristic polynomial lie in the unit circle.  On the other hand,  the map \(P_{A_3}^d\)  causes chaos in the region  \(W_3\)  that corresponds to the case that all the zeros of its self-inversive   polynomial lie on the unit circle.

Peterson and Sinclair \cite{PeS} and Sinclair and Vaaler \cite{SV} studied monic  self-reciprocal polynomials.  Those polynomials have appeared in the study of random polynomials, random matrix theory and number theory.  

\subsection{Summary of results }\hspace{0cm}

We consider complex binary forms expressed by
\[P(z,w) =  \sum_{j=0}^{n+1}\zeta_jz^{n+1-j}w^j.\]
Let \({\mathcal A}_n \) denote the set of the   self-inversive forms \(P(z, w)\)  satisfying that  (\(n\) is odd and \(\mu = 1\)) or   (\(n\) is even  and \(\mu = -1\)) ,  where  \(\mu\) is the constant in (1.1).   That is,
\[{\mathcal A}_n : = \{P(z,w) : \zeta_j = (-1)^{n+1}\overline{\zeta}_{n+1-j}, \enskip \mbox{for}\enskip j = 0, 1, \dots , n+1\}.\]
Let \({\mathcal B}_n \) denote the set of the   self-inversive forms \(P(z, w)\)  satisfying that  (\(n\) is odd and \(\mu = -1\)) or   (\(n\) is even  and \(\mu = 1\)) .  That is, 
\[{\mathcal B}_n : = \{P(z,w) : \zeta_j = (-1)^{n}\overline{\zeta}_{n+1-j}, \enskip \mbox{for}\enskip j = 0, 1, \dots , n+1\}.\]

We also consider real binary forms
\[g(X,Y) =  \sum_{j=0}^{n+1}a_jX^{n+1-j}Y^j, \enskip a_0, a_1, \dots , a_{n+1} \in \mathbb{R}.\]
Let \({\mathcal F}_n \) denote the set of the real binary forms \(g(X, Y)\).

We will show that there exist an isomorphism from  \({\mathcal A}_n \)   to \({\mathcal F}_n \) and  another  isomorphism from  \({\mathcal B}_n \)   to \({\mathcal F}_n \).

Let  \({\mathcal {MA}}_n \) denote the set of  monic self-reciprocal forms in \({\mathcal A}_n \) satisfying  \(\zeta_0 =  1\).
Let  \({\mathcal {MF}}_n \) denote the set of the corresponding   forms in \({\mathcal F}_n \).  We will show that for any form \(g(X,Y) \) in  \({\mathcal {MF}}_n \) satisfying \(a_0 \ne 0\), the discriminant of the polynomial \(g(X,1)\) is expressed by the determinant of a \((n+1)\times (n+1)\) matrix.  Any element of the matrix is a polynomial in \(\zeta_1,  \zeta_{2} , \dots , \zeta_n\) with integer coefficients. The matrix is a slightly modified version of the matrix introduced by \cite{EL}.  Their matrices are related to generalized Chebyshev polynomials of the first kind in several complex variables.

Let  \({\mathcal {MB}}_n \) denote the set of monic self-reciprocal forms in \({\mathcal {B}}_n \) for  even \(n\).  The same result holds for \({\mathcal {MB}}_n \).

\section{Results and proofs}
\subsection{The spaces \({\mathcal A}_n \)  and  \({\mathcal F}_n \) }\hspace{0cm}

We consider self-inversive forms in the form

\begin{equation} 
 f_{n}(T, U) = \zeta_0T^{n+1} - \zeta_1T^{n}U + \zeta_2T^{n-1}U^2 + \cdot \cdot \cdot + (-1)^n\zeta_nTU^{n} + (-1)^{n+1}\zeta_{n+1}U^{n+1}.
\end{equation}
where 
\begin{equation} 
\zeta_k \in  {\mathbb C}, \quad \zeta_k = {\bar \zeta}_{n+1-k}, \quad k = 0,  \cdot \cdot \cdot , n+1.
\end{equation}
 Let \({\mathcal A}_n \) denote the space of these self-inversive forms.  Note that if \(n\) is odd,  then 
\(f_{n}(T, U)\) is self-reciprocal and that if \(n\) is even,  then 
\(f_{n}(T, -U)\) is self-reciprocal.  The space  \({\mathcal A}_n \)  is related to a complex Lie algebra of type \(A_n \) .   See \cite{U1} and \cite{U2}.

We also consider binary forms in the forms   
\begin{equation} 
\begin{split}
 g_{n}(X, Y) = a_0X^{n+1} + a_1X^{n}Y +  \cdot \cdot \cdot + a_nXY^{n} + a_{n+1}Y^{n+1}, \\
\quad \mbox{with} \quad a_0, a_1,  \cdot \cdot \cdot , a_{n+1} \in {\mathbb R}.
\end{split}
\end{equation}
Let   \({\mathcal F}_n \) denote the space of these forms.  We do not assume that \(a_0 \ne 0\).
We  show that  \({\mathcal A}_n \) is isomorphic to  \({\mathcal F}_n \).

\begin{pro} \label{pro:2.1}
There exists an isomorphism  \( \Phi_n\) from  \({\mathcal A}_n \) onto  \({\mathcal F}_n \).
\end{pro}

We define the map  \( \Phi_n\) as follows.  Clearly \(\overline{ f_n(T, -{\overline T})} = f_n(T, -{\overline T}).\)  Then we set

\begin{equation} 
T = i(X+iY) \quad  \mbox{and} \quad    U = i(X-iY) .  
\end{equation}
Hence \(f_n  (i(X+iY),  i(X-iY)) \)  is a polynomial in variables \(X\) and \(Y\) with real coefficients.  So  \(f_n  (i(X+iY) , i(X-iY)) \)   is written as \(g_n(X,Y)\) in (2.3).  The image of a self-inversive form (2.1) under the map  \( \Phi_n\)  is the binary form \(g_n(X, Y)\). The coefficients \(a_k\) are given by the following proposition.

 \begin{pro} \label{pro:2.2}
(1) If \(n\) is even, then 
\[a_{2h+1} = (-1)^{n/2+h+1} {{n+1}\choose{2h+1}}(\zeta_0 + \zeta_{n+1})\]
\[+ (-1)^{n/2+1} \sum_{k=1}^{n/2}\left[(-1)^k\Bigg\{\sum_{j=0}^{h}(-1)^j{{k}\choose{h-j}}{{n+1-2k}\choose{2j+1}}\Bigg\}(\zeta_k + \zeta_{n+1-k})\right],\]
and 
\[a_{2h} = (-1)^{n/2+h} {{n+1}\choose{2h}}(\zeta_0 - \zeta_{n+1})\sqrt{-1}\]
\[+ (-1)^{n/2} \sum_{k=1}^{n/2}\left[(-1)^k\Bigg\{\sum_{j=0}^{h}(-1)^j{{k}\choose{h-j}}{{n+1-2k}\choose{2j}}\Bigg\}(\zeta_k - \zeta_{n+1-k})\sqrt{-1}\right],\]
  \[\mbox{where} \quad h = 0, 1, \dots ,  n/2.\qquad \qquad \qquad \qquad\qquad \qquad \qquad \qquad\qquad \qquad \qquad \qquad\]
(2) If \(n\) is odd, then 
\[a_{2h+1} = (-1)^{(n+1)/2+h} {{n+1}\choose{2h+1}}(\zeta_0 - \zeta_{n+1})\sqrt{-1}\]
\[+ (-1)^{(n+1)/2} \sum_{k=1}^{(n-1)/2}\left[(-1)^k\Bigg\{\sum_{j=0}^{h}(-1)^j{{k}\choose{h-j}}{{n+1-2k}\choose{2j+1}}\Bigg\}(\zeta_k - \zeta_{n+1-k})\sqrt{-1}\right],\]
  \[\mbox{where} \quad h = 0, 1, \dots ,  (n-1)/2,\qquad \qquad \qquad \qquad\qquad \qquad \qquad \qquad\qquad \qquad \qquad \qquad\] 
and
\[a_{2h} = (-1)^{(n+1)/2+h} {{n+1}\choose{2h}}(\zeta_0 + \zeta_{n+1}) + {{(n+1)/2}\choose{h}}\zeta_{(n+1)/2}\]
\[+ (-1)^{(n+1)/2} \sum_{k=1}^{(n-1)/2}\left[(-1)^k\Bigg\{\sum_{j=0}^{h}(-1)^j{{k}\choose{h-j}}{{n+1-2k}\choose{2j}}\Bigg\}(\zeta_k + \zeta_{n+1-k})\right],\]
  \[\mbox{where} \quad h = 0, 1, \dots ,  (n+1)/2.\qquad \qquad \qquad \qquad\qquad \qquad \qquad \qquad\qquad \qquad \qquad \qquad\]

Here we use conventions 
\[{k \choose 0} = 1, \enskip \mbox{and } \enskip {k \choose j} = 0,   \enskip \mbox{if } \enskip j > k \geq 0.\]
\end{pro}
\begin{proof} [Proof of Proposition 2.2.]\quad 
 Case 1 :  \quad \(n\) is even.\\
 We consider pairs of terms of 
 \(f_n(i(X+iY), i(X+iY))\).  Set
\[Q_k : = i^{n+1}(-1)^k(X^2+Y^2)^k\{\zeta_k (X+iY)^{n+1-2k}-\zeta_{n+1-k} (X-iY)^{n+1-2k}\},\]
\[k = 1, \dots , [n/2].\]
  Easy calculations reveal that the coefficient of  \(X^{n-2h}Y^{2h+1}\) in \(Q_k\) is 
\[ (-1)^{n/2+k+1} \Bigg\{\sum_{j=0}^{h}(-1)^j{{k}\choose{h-j}}{{n+1-2k}\choose{2j+1}}\Bigg\}(\zeta_k + \zeta_{n+1-k})\]
and  the coefficient of  \(X^{n+1-2h}Y^{2h}\) in \(Q_k\) is
\[ (-1)^{n/2+k} \Bigg\{\sum_{j=0}^{h}(-1)^j{{k}\choose{h-j}}{{n+1-2k}\choose{2j}}\Bigg\}(\zeta_k - \zeta_{n+1-k})\sqrt{-1}.\]
To get \(a_{2h+1}\) and \(a_{2h}\),  we sum the coefficients of  \(X^{n-2h}Y^{2h+1}\)   and  \(X^{n+1-2h}Y^{2h}\)  \enskip in \(Q_k\)  over       \(k = 1, \dots ,  n/2\) \enskip and those coefficients of the pair of terms \enskip \(\zeta_0 (X+iY)^{n+1}-\zeta_{n+1} (X-iY)^{n+1}\),  respectively.\\
 Case 2 :  \quad 
\(n\) is odd.  By the similar method, we can prove the assertion. 
\end{proof}

Conversely we consider the inverse of the map \(\Phi_n\).  From (2.4), we have
\begin{equation} 
X = -\frac i2(T+U) \quad \mbox{and} \quad Y = -\frac 12(T - U).
\end{equation}
Then we need to show
\begin{equation*} 
g_n ( -\frac i2(T+U), -\frac 12(T - U)) \in {\mathcal A}_n.
\end{equation*}
We set
\begin{equation*} 
f(T, U) : = g_n ( -\frac i2(T+U), -\frac 12(T - U)).
\end{equation*}
Then it suffices to prove the equality
\begin{equation} 
f(T, U) = (-1)^{n+1}\hat{f}(U, T)  ,
\end{equation}
where, for a form
\[f(T, U)  =  \sum_{j=0}^{n+1}\zeta_jT^{n+1-j}U^j,\]
we define the form \(\hat{f}(T, U)\) by  
\[\hat{f}(T, U)  =  \sum_{j=0}^{n+1}\bar{\zeta}_jT^{n+1-j}U^j.\]

To prove this we note that 
\begin{equation*} 
\begin{split}
 f(U, T)
 = a_0(-\frac i2(U+T))^{n+1} +  a_1(-\frac i2(U+T))^{n}(-\frac 12(U-T)) + \dots +  a_{n+1}(-\frac12(U-T))^{n+1}\\
= a_0(-\frac i2(U+T))^{n+1} -  a_1(-\frac i2(U+T))^{n}(-\frac 12(T-U)) + \dots +  (-1)^{n+1}a_{n+1}(-\frac12(T-U))^{n+1}.
\end{split}
\end{equation*}
\[ \hat{f}(U, T)
 =  (-1)^{n+1}a_0(-\frac i2(U+T))^{n+1} +   (-1)^{n+1}a_1(-\frac i2(U+T))^{n}(-\frac 12(T-U)) \]
\[+ \dots +   (-1)^{n+1}a_{n+1}(-\frac12(T-U))^{n+1} =  (-1)^{n+1}f(T, U).\]
Then we have (2.6) and so \(\Phi_n\) is a bijection.  Since  \(\Phi_n\) is a linear map, we have
 \[\Phi_n(f_n + f'_n) =  \Phi_n(f_n) +  \Phi_n(f'_n), \enskip\mbox{where} \enskip f_n,  f'_n \in {\mathcal A}_n.\]
Hence \(\Phi_n\)  is an isomorphism.\\
\\

Next we give an alternative proof of Proposition 2.1.  This is useful in the sections below.

Set \(Z = X + iY\).  Let
\[f_n(Z) := f_n(iZ, i\overline{Z}) = f_n(i(X+iY), i(X-iY)).\]
Then \(f_n(Z)\)  is written as follows.
\\
\begin{lemma} \label{lemma:2.3}
If  \(n\) is even, then 
\begin{equation} 
 f_{n}(Z) = 2Re\enskip i^{n+1}(\zeta_0Z^{n+1} - \zeta_1Z^{n}\overline{Z} + \dots +  (-1)^{n/2}\zeta_{n/2}Z^{(n+2)/2}\overline{Z}^{n/2}  ).
\end{equation}
If  \(n\) is odd, then 
\begin{equation*} 
\begin{split}
 f_{n}(Z) = 2Re \enskip i^{n+1}(\zeta_0Z^{n+1} - \zeta_1Z^{n}\overline{Z} + \dots +   (-1)^{(n-1)/2}\zeta_{(n-1)/2}Z^{(n+3)/2}\overline{Z}^{(n-1)/2} \\
+ \frac { (-1)^{(n+1)/2}}2 \zeta_{(n+1)/2}(Z\overline{Z})^{(n+1)/2} ).
\end{split}
\end{equation*}

\end{lemma}
\begin{proof}\quad
 Case 1 :  \quad \(n\) is even.  Set
\[F(T, -\overline{T}) = \zeta_0T^{n+1} + \zeta_1T^{n}\overline{T} + \dots +  \zeta_{n/2}T^{(n+2)/2}\overline{T}^{n/2} .\]
Then
\[f_n(T, -\overline{T}) = F(T, -\overline{T}) +  \overline{F(T, -\overline{T})} = 2Re \enskip F(T, -\overline{T}).\]
Setting  \(T = iZ\), \enskip we have (2.7).\\
 Case 2 :  \quad 
\(n\) is odd.  The proof is similar.
\end{proof}
Let
\[w_j : = Z^{n+1-i}\overline{Z}^{j}, \quad w_j \in {\mathbb C}[X, Y], \quad j = 0, \dots , n+1.\]
We regard \(w_j\) as a polynomial in \( {\mathbb C}[X, Y]\).\\
Set 
\[{\mathcal S}_1: = \{w_j(X, Y) :  j = 0, \dots , n+1\}.\]
Probably, the following lemma may have been known.  But, here we give a proof of it.
\begin{lemma} \label{lemma:2.4}
(1)  The family \({\mathcal S}_1\) is linearly independent over \( {\mathbb C}\).\\
(2) We define real polynomials  \(R_j(X,Y)\)  and    \(I_j(X,Y)\) by 
\[w_j(X,Y) = Z^{n+1-i}\overline{Z}^{j} = R_j(X,Y) + iI_j(X,Y),  \quad R_j(X,Y),  I_j(X,Y)  \in {\mathbb R}[X, Y].\]
(a)  If  \(n\) is even, then 
\[R_j = R_{n+1-j}, \quad  I_j = -I_{n+1-j}, \quad j = 0, \dots , n/2.\]
 If  \(n\) is odd, then 
\[R_j = R_{n+1-j}, \quad  I_j = -I_{n+1-j}, \quad j = 0, \dots , (n-1)/2, \]
\[ I_{(n+1)/2} = 0.\]
(b)  Set
\begin{equation*}
\begin{split}
{\mathcal S}_2 : = \left \{
   \begin{array}{lll}
   \{R_0, I_0, \dots , R_{n/2},  I_{n/2} \} \quad \mbox{if n is even}, \\
\\
  \{R_0, I_0, \dots , R_{(n-1)/2},  I_{(n-1)/2}, R_{(n+1)/2}\}\quad \mbox{if n is odd}.
   \end{array}
   \right.
\end{split}
\end{equation*}
Then the  family \({\mathcal S}_2\) is linearly independent over \( {\mathbb C}\).\\
\end{lemma}
\begin{proof}\quad
(1)  \quad We define a Hermitian product on \({\mathcal S}_1\) by
\[<w_j, w_k> = \frac {n+2}{\pi} \int\limits_{|X^2+Y^2| \le 1} w_j\overline{w}_kdXdY.\]
Then elements \(w_j(X,Y)\)  are orthogonal on the unit disk.  Hence the  family \({\mathcal S}_1\) is linearly independent over \( {\mathbb C}\).\\
(2)  The assertion (a) is trivial.\\
(b)\quad  Case 1 :  \quad \(n\) is even.  Note that 
\[R_j = \frac 12(w_j + w_{n+1-j}), \quad I_j = \frac 1{2i}(w_j - w_{n+1-j}), \quad j = 0, \dots , n/2. \]
Since the family \(\{w_j\}\) is  linearly independent, the  family \({\mathcal S}_2\) is linearly independent.\\
 Case 2 :  \quad 
\(n\) is odd.  By the similar method we can prove the assertion (b).
\end{proof}

We recall \( f_{n}(T, U)\)  and the definition of the map \(\Phi_n\).
\begin{equation*} 
 f_{n}(T, U) = \zeta_0T^{n+1} +  (-1)\zeta_1T^nU +\dots + (-1)^n\zeta_nTU^{n} + (-1)^{n+1}\zeta_{n+1}U^{n+1}.
\end{equation*}
Substituting
\[T = i(X+iY) \quad \mbox{and} \quad U = i(X-iY).\]
in  \(f_n(T, U)\) we obtain a real polynomial 
\begin{equation*} 
 g_{n}(X, Y) = a_0X^{n+1} + a_1X^{n}Y +  \cdot \cdot \cdot + a_nXY^{n} + a_{n+1}Y^{n+1}.
\end{equation*}
Let
\[\zeta_j = \xi_j + i\eta_j,\enskip \xi_i, \eta_j \in {\mathbb R}, \quad j = 0, 1, \dots [(n+1)/2].\]
Note that if \(n\) is odd, then  \(\eta_{(n+1)/2} = 0.\)

We denote by \(v_1\) the transpose of the row vector  \((a_0, a_1,  \dots , a_{n+1}).\)   We denote by \(v_2\) the transpose of the row vector  \((\xi_0, \eta_0, \xi_1, \eta_1, \dots , \xi_{n/2}, \eta_{n/2})\)  if \(n\) is even, or the row vector 
 \((\xi_0, \eta_0, , \dots , \xi_{(n-1)/2}, \eta_{(n-1)/2}, \xi_{(n+1)/2})\)  if \(n\) is odd.

\begin{lemma} \label{lemma:2.5}
There exists an invertible matrix \(M\) satisfying
\[v_1 = Mv_2.\]
\end{lemma}
\begin{proof}
\quad  Case 1 :  \quad \(n\) is even.  From Lemma 2.3, we know that 
\begin{equation*}
\begin{split}
f_n(i(X+iY), i(X-iY)) &
 = 2(-1)^{n/2+1}\{(\xi_0I_0 + \eta_0R_0) + \dots \\
&+  (-1)^{n/2}(\xi_{n/2}I_{n/2} + \eta_{n/2}R_{n/2})\}.\\
                         & = a_0X^{n+1} +  a_1X^{n}Y + \dots +  a_{n+1}Y^{n+1} .
\end{split}
\end{equation*}
We denote by \(V\) the vector space  over \( {\mathbb C}\) generated by homogeneous polynomials \(X^{n+1} , X^{n}Y, \dots .\) and 
\(Y^{n+1} \).  We see from Lemma 2.4 2(b) that  \(\{R_0, I_0, \dots , R_{n/2},  I_{n/2} \}\)  is another base of \(V\).  Then there exists an invertible matrix \(M\) satisfying 
\[v_1 = Mv_2.\]
 Case 2 :  \quad 
\(n\) is odd.  The proof is similar.
\end{proof}
\begin{proof}[Proof of Proposition 2.1.]\quad
From Lemma 2.5, we conclude that the map \(\Phi_n\) is bijective.  
\end{proof}

Next we consider the relation between the roots of  \(f_n(T, U)\) and those of the corresponding form  \(g_n(X, Y)\) under \(\Phi_n\).  Let
\begin{equation} 
\begin{split}
 f_n(T, U) = \mathop \Pi_{j=1}^{n+1}(u_jT - t_jU),\\
 g_n(X, Y) = \mathop \Pi_{j=1}^{n+1}(y_jX - x_jY).
\end{split}
\end{equation}
We recall (2.5) :
\begin{equation*} 
X = -\frac i2(T+U) \quad \mbox{and} \quad Y = -\frac 12(T - U).
\end{equation*}
We define a map \(\varphi\) from \( {\mathbb P}^1({\mathbb C})\)  to \( {\mathbb P}^1({\mathbb C})\)  by 
\[ \varphi(T : U) : = (X : Y)  = (-\frac i2(T+U) : -\frac 12(T - U)).\]
Then with a suitable numbering
\[\varphi(t_j : u_j) = (x_j : y_j) = (i(t_j + u_j) : t_j - u_j).\]
The map \(\varphi\) is a  M\(\ddot{o}\)bius transformation on  \( {\mathbb P}^1\).\\
Clearly \enskip  \(y_j = 0\)\enskip if and only if \enskip \(t_j = u_j\).\\
If \enskip  \(y_j \ne 0\),\enskip we may set  \enskip \(y_j = u_j = 1\).  Then if \enskip  \(t_j \ne 1,\)\\
\begin{equation}
 x_j = \sqrt{-1}\frac{t_j+1}{t_j-1}.
\end{equation}
If  \enskip  \(y_j = 0\),\enskip then  \enskip \(\varphi^{-1}(1 : 0) = (1 : 1)\).\\
The point (1 : 0) is the point at infinity and the point  (1 : 1) lies on the unit circle \(C\).
The real axis of  \( {\mathbb P}^1\) maps into the unit circle under the  M\(\ddot{o}\)bius transformation  \(\varphi^{-1}\).  Then we have the following proposition. 
 \begin{pro} \label{pro:2.6}
We suppose \enskip \(y_j = u_j = 1\) \enskip and  \enskip  \(t_j \ne 1\) .\\
(1) \(x_j \)\enskip is real  if and only if \enskip \(\mid t_j \mid = 1\).\\
(2) \(x_j, x_{j+1} \notin   {\mathbb R}\) \enskip and \enskip  \(\overline{x_j} = x_{j+1}\) \\
  if and only if \enskip \( t_j = re^{i\theta}\) \enskip and \enskip \( t_{j+1} = \frac 1re^{i\theta}\) \enskip with \enskip \(r \ne 1\).
\end{pro}
 The pair \((t_j : 1)\)  and  \((t_{j+1} : 1)\)  is said to be \textit{symmetric in the unit circle C} if it satisfies that \enskip \( t_j = re^{i\theta}\) \enskip and \enskip \( t_{j+1} = \frac 1re^{i\theta}\) \enskip with \enskip \(r \ne 1\) . 
\subsection{The spaces  \({\mathcal B}_n \)  and  \({\mathcal F}_n \)}\hspace{0cm}

In this section we  will give an isomorphism \(\Psi_n\) from  the space \({\mathcal B}_n \)  onto  \({\mathcal F}_n \) .  We divide   \({\mathcal B}_n \)   into two spaces   \({\mathcal B}(e, 1) \)  and  \({\mathcal B}(o, -1) \).  Let   \({\mathcal B}(e, 1) \)  denote the space of the self-reciprocal forms :
\begin{equation} 
p_{n}(T, U) = \zeta_0T^{n+1} + \zeta_1T^{n}U + \zeta_2T^{n-1}U^2 + \cdot \cdot \cdot + \zeta_nTU^{n} + \zeta_{n+1}U^{n+1},
\end{equation}
where \(n\) is even and 
\begin{equation*} 
\zeta_k \in  {\mathbb C}, \quad \zeta_k = {\bar \zeta}_{n+1-k}, \quad k = 0,  \cdot \cdot \cdot , n+1.
\end{equation*} 
Let   \({\mathcal B}(o, -1) \)  denote the space of the self-inversive forms :
\begin{equation*} 
p_{n}(T, U) = \zeta_0T^{n+1} + \zeta_1T^{n}U + \zeta_2T^{n-1}U^2 + \cdot \cdot \cdot + \zeta_nTU^{n} + \zeta_{n+1}U^{n+1},
\end{equation*}
where \(n\) is odd and 
\begin{equation*} 
\zeta_k \in  {\mathbb C}, \quad \zeta_k = -{\bar \zeta}_{n+1-k}, \quad k = 0,  \cdot \cdot \cdot , n+1.
\end{equation*} 
Since  \(\zeta_{(n+1)/2} = -{\bar \zeta}_{(n+1)/2}, \)  \enskip we may set 
 \[\zeta_{(n+1)/2} = i{ \tau}_{(n+1)/2}, \enskip   \tau_{(n+1)/2} \in  {\mathbb R}.\]

In the first place we study the space  \({\mathcal B}(e, 1) \) .  Clearly
\[p_n(T, \overline{T}) = \overline{p_n(T, \overline{T})}.\]
Then we set 
\begin{equation} 
T = X + iY \enskip \mbox{ and } \enskip U = X -iY.
\end{equation}
We denote the real polynomial  \(p_n( X + iY,   X -iY)\)  by
\begin{equation*} 
\begin{split}
 q_{n}(X, Y) = b_0X^{n+1} + b_1X^{n}Y +  \cdot \cdot \cdot + b_nXY^{n} + b_{n+1}Y^{n+1}, \\
\quad \mbox{with} \quad b_0, b_1,  \cdot \cdot \cdot , b_{n+1} \in {\mathbb R}.
\end{split}
\end{equation*}
We define the map  \(\Psi_n\) by  \(\Psi_n(p_n) = q_n.\)  Set  \(Z = X + iY.\)   Then as in Lemma 2.3 we have
\begin{equation*} 
 p_{n}(Z) = 2Re\enskip (\zeta_0Z^{n+1} + \zeta_1Z^{n}\overline{Z} + \dots +  \zeta_{n/2}Z^{(n+1)/2}\overline{Z}^{n/2}  ),
\end{equation*}

Next we consider  the space  \({\mathcal B}(o, -1) \) .    Set    \(\epsilon = e^{\pi i/(2(n+1))}.\)
Clearly
\[p_n(\epsilon Z,  \epsilon\overline{Z}) = \overline{p_n(\epsilon T, \epsilon\overline{Z})}.\]
Then we set 
\begin{equation} 
T = \epsilon(X + iY) \enskip \mbox{ and } \enskip U = \epsilon(X -iY).
\end{equation}
We  denote the real polynomial  \(p_n( \epsilon (X + iY),   \epsilon (X -iY))\)  by
\begin{equation*} 
\begin{split}
 q_{n}(X, Y) = b_0X^{n+1} + b_1X^{n}Y +  \cdot \cdot \cdot + b_nXY^{n} + b_{n+1}Y^{n+1}, \\
\quad \mbox{with} \quad b_0, b_1,  \cdot \cdot \cdot , b_{n+1} \in {\mathbb R}.
\end{split}
\end{equation*}
We define the map  \(\Psi_n\) by  \(\Psi_n(p_n) = q_n.\)  Set  \(Z = X + iY.\)   Then
\begin{equation*} 
\begin{split}
 p_{n}(Z) = -2Im\enskip (\zeta_0Z^{n+1} + \zeta_1Z^{n}\overline{Z} + \dots \\
+  \zeta_{(n-1)/2}Z^{(n+3)/2}\overline{Z}^{(n-1)/2}  ) - \tau_{(n+1)/2}(Z\overline{Z})^{(n+1)/2}.
\end{split}
\end{equation*}

Hence in both cases,  by the same proof as the alternative proof  of Proposition 2.1,  we can show the following proposition.
\begin{pro} \label{pro:2.7}
There exists an isomorphism  \( \Psi_n\) from  \({\mathcal B}_n \) onto  \({\mathcal F}_n \).
\end{pro}

Next we consider the relation between the roots of  \(p_n(T, U)\) and those of the corresponding form  \(q_n(X, Y)\) under \(\Psi_n\).  Let
\begin{equation} 
\begin{split}
 p_n(T, U) = \mathop \Pi_{j=1}^{n+1}(u'_jT - t'_jU),\\
 q_n(X, Y) = \mathop \Pi_{j=1}^{n+1}(y'_jX - x'_jY).
\end{split}
\end{equation}
By (2.11) and (2.12), in both cases we may define 
 a map \(\psi\) from \( {\mathbb P}^1({\mathbb C})\)  to \( {\mathbb P}^1({\mathbb C})\)  by 
\[\psi(t'_j : u'_j) = (i(t'_j + u'_j) : t'_j - u'_j).\]
We assume that   \enskip \(y'_j = u'_j = 1\) \enskip and \enskip \(t'_j \ne 1\).  Then 
\begin{equation}
 x'_j = \sqrt{-1}\frac{t'_j+1}{t'_j-1}.
\end{equation}
This is the same as that in (2.9).

\subsection{Discriminants of monic self-inversive polynomials in \({\mathcal A}_n \)}\hspace{0cm}

We consider  monic self-inversive polynomials in \({\mathcal A}_n \) and their discriminant.  Let  \({\mathcal {MA}}_n \)  denote  self-inversive forms that are written in (2.1) satisfying  \(\zeta_0 = \zeta_{n+1} = 1.\)    The image of \({\mathcal {MA}}_n \) under 
 \(\Phi_{n}\) is denoted by \(\mathcal {MF}_n \).  Any element of \({\mathcal {MF}}_n \)  is written as
\begin{equation} 
\begin{split}
 \tilde{g}_{n}(X, Y) = c_0X^{n+1} + c_1X^{n}Y +  \cdot \cdot \cdot + c_nXY^{n} + c_{n+1}Y^{n+1}, \\
c_j \in {\mathbb R}, \quad j = 0, 1, \dots, n+1.
\end{split}
\end{equation}

In general, the general linear group \(GL(2,  {\mathbb R})\) acts on the space of binary forms \({\mathcal F}_n \)  via linear substitution :
\[(X, Y) \to (X, Y)M =  (\alpha X + \gamma Y, \beta X + \delta Y),\]
\[\mbox{where} \quad M = \left( 
\begin{array}{cc}
\alpha &  \beta \\
\gamma & \delta
\end{array} \right); \quad \mbox{that is },  \quad {M\enskip maps}
\]
\[g(X, Y) \to g^M(X, Y) =  g(\alpha X + \gamma Y,  \beta X + \delta Y).\]

\begin{lemma} \label{lemma:2.8}
Suppose an element \(g_n\) in  \({\mathcal F}_n \)
 satisfies that   \(\Phi_n^{-1}(g_n)\) is written as
\begin{equation*} 
 \zeta_0T^{n+1} - \zeta_1T^{n}U + \zeta_2T^{n-1}U^2 + \cdot \cdot \cdot + (-1)^n\zeta_nTU^{n} + (-1)^{n+1}\zeta_{n+1}U^{n+1},
\end{equation*}
with \quad  \(\zeta_0 \ne 0.\)\\
Then there exists a matrix \(M\) in \(GL(2,  {\mathbb R})\)  satisfying 
\[g_n^M \in {\mathcal {MF}}_n.\]
\end{lemma}
\begin{proof}\quad
From Lemma 2.3, we know that
\begin{equation*} 
 g_{n}(X, Y) = 2Re\enskip i^{n+1}(\zeta_0Z^{n+1} - \zeta_1Z^{n}\overline{Z} + \dots +   (-1)^{n/2}\zeta_{n/2}Z^{(n+1)/2}\overline{Z}^{n/2}  ).
\end{equation*}
Let \[\zeta _0 = (re^{i\theta})^{n+1}, \quad r > 0, \quad 0 \le \theta < 2\pi.\]
To replace \enskip \(\zeta _0 Z^{n+1}\)  by  \((Z')^{n+1}\), we set 
\[\left( 
\begin{array}{c}
X' \\
Y'
\end{array} \right)
= r\left( 
\begin{array}{cc}
\cos\theta &  -\sin\theta \\
\sin\theta & \cos\theta
\end{array} \right)
\left( 
\begin{array}{c}
X \\
Y
\end{array} \right),
\]
where \quad
\(Z = X +iY, \quad Z' = X' + iY'.\)  Then we have new coefficients \enskip \(\zeta'_0 = 1,  \zeta'_1,  \dots ,  \zeta'_{n/2}  \) .
Hence by (2.2) we get new coefficients   \(\zeta'_0,  \zeta'_1, \dots , \zeta'_{n+1}\). 
The same folds when  \(n\)  is odd.
\end{proof}

Note that this transformation is not unique.

Next we define the discriminant of \( \tilde{g}_{n}(X, 1)\) in (2.15).  To define the discriminant of \( \tilde{g}_{n}(X, 1)\) we assume that  \(c_0 \ne 0\).  We define the discriminant by

\begin{equation}
Dis(\tilde{g}_{n}(X, 1)) = c_0^{2n} \mathop{\Pi}  _{1 \le j < k \le n+1}(x_j - x_k)^2,
\end{equation}
where  \(x_1, x_2, \dots , x_{n+1}\)  are roots of \( \tilde{g}_{n}(X, 1)\).   Though  Gelfand,   Kapranov  and  Zelevinsky  \cite{GKZ}  defined  \(Dis(\tilde{g}_{n}(X, 1))\)  to be  \((-1)^{n(n+1)/2}\)  times the right-hand side of (2.16),  we adopt the definition in (2.16).  In our definition,    if all the roots of \enskip
  \(\tilde{g}_{n}(X, 1)\)   are real, then  we have \enskip  \(Dis(\tilde{g}_{n}(X, 1)) \geq 0\) .  This is convenient to Corollary 2.10 below.

In this section, we consider monic self-inversive polynomials in  \({\mathcal A}_n \).  We denote the image of  \( \tilde{g}_{n}(X, Y)\)   in  (2.15) under the map \(\Phi_n^{-1}\) by 
\begin{equation} 
 \tilde{f}_{n}(T, U) = T^{n+1} - \zeta_1T^{n}U + \zeta_2T^{n-1}U^2 + \cdot \cdot \cdot + (-1)^n\zeta_nTU^{n} + (-1)^{n+1}U^{n+1}.
\end{equation}
We consider the roots of  \( \tilde{f}_{n}\)  and \( \tilde{g}_{n}\).  We use the notation in (2.8).  In this case we may set   \(u_j = y_j = 1\)  for   \(j =1, \dots , n+1\).     Hence  \(x_1, x_2, \dots , x_{n+1}\)  are roots of \( \tilde{g}_{n}(X, 1)\)  and   \(t_1, t_2, \dots, t_{n+1} \) are roots of   \( \tilde{f}_{n}(T, 1)\).

We will define a \((n+1)\times(n+1)\)  matrix  \(H_n\)  that is a slightly modified version of the matrix introduced by Eier and Lidl \cite{EL}.
Let  \(t_1,  1 \le j \le n+1,\) be elements in \({\mathbb C}\) satisfying \(t_1t_2 \dots t_{n+1} = 1\).  The coefficient   \(\zeta_j\) is the \(j\)-th elementary symmetric function in \(t_1,\dots, t_{n+1}\).  We define the \(m\)-th power \(h_m(m \in {\mathbb Z})\)  by
\[h_m(\zeta_1, \dots , \zeta_n) :  =  \sum_{j=1}^{n+1}t_j^m.\]
Since \enskip \(t_1t_2 \dots t_{n+1} = 1\),  \enskip we may regard \(h_{m}\) as a polynomial in  \enskip \(\zeta_1, \zeta_2, \dots , \zeta_n\)  even when  \enskip \(m < 0\).\\

We define a  \((n+1)\times(n+1)\)  matrix  \(H_n\)  by
\[H_n :  = \left(
\begin{array}{cccc}
h_0 & h_{-1} &\dots & h_{-n}\\
h_1 & h_{0} &\dots & h_{-n+1}\\
\vdots &   &\ddots & \vdots\\
h_n & h_{n-1} &\dots & h_{0}
\end{array}
\right).\]
In our case 
\[\zeta_j = \overline{\zeta}_{n+1-j} , \enskip j = 1, \dots , n.\]
     Eier and Lidl \cite{EL} considered only the case that any \(t_j\)  lies on the unit circle \(C\).

The determinant  \(det \enskip H_n\) is a polynomials in \(\zeta_1, \dots , \zeta_n\).  The relation between the discriminant \(Dis( \tilde{g}_{n}(X, 1))\)  in (2.16) and the  determinant  \(det\enskip H_n\) is shown in the following theorem.
\begin{theorem} \label{theorem:2.9}
Under the above notations we have 
\[Dis( \tilde{g}_{n}(X,1)) = 2^{n(n+1)}det\enskip H_n.\] 
\end{theorem}
\begin{proof}
\quad  Case 1 :  \quad \(n\) is even.  Using the formula for the Vandermonde matrix we  see 
\[det\enskip H_n   = \left|
\begin{array}{cccc}
1/t_{1}^{n} & 1/t_{2}^{n} &\dots & 1/t_{n+1}^{n}\\
1/t_{1}^{n-1} & 1/t_{2}^{n-1} &\dots & 1/t_{n+1}^{n-1}\\
\vdots & \vdots  &   & \vdots\\
1 & 1  &\dots & 1
\end{array}
\right|
\left|
\begin{array}{cccc}
t_{1}^{n} & t_{1}^{n-1} &\dots & 1\\
t_{2}^{n} & t_{2}^{n-1} &\dots & 1\\
\vdots & \vdots  &   & \vdots\\
t_{n+1}^{n} & t_{n+1}^{n-1} &\dots & 1
\end{array}
\right|
\]
\[= \prod_{1 \le j < k \le n+1}(\frac 1{t_j}-\frac 1{t_k})\times \prod_{1 \le j < k \le n+1}({t_j}-{t_k})\]
\[= (-1)^{n(n+1)/2}\prod_{1 \le j < k \le n+1}({t_j}-{t_k})^2 \times \frac 1{(t_1t_2\dots t_{n+1})^n}\]
\[= (-1)^{n(n+1)/2}\prod_{1 \le j < k \le n+1}({t_j}-{t_k})^2.\]

On the other hand by (2.16), 
\[Dis( \tilde{g}_{n}) = c_0^ {2n} \times \prod_{1 \le j < k \le n+1}({x_j}-{x_k})^2.\] 
From (2.9) it follows that
\[{x_j}-{x_k} = \frac {2i(t_k-t_j)}{(t_j-1)(t_k-1)}.\] 
Hence
\begin{equation} 
Dis( \tilde{g}_{n}) 
=  (-1)^{n(n+1)/2}c_0^ {2n} \times 2^{n(n+1)}\times\prod_{j < k }({t_j}-{t_k})^2/\prod_{j =1 }^{n+1}(1-{t_j})^{2n}.
\end{equation}
Clearly
\[\prod_{j =1 }^{n+1}(1-{t_j}) = \tilde{f}_{n}(1, 1).\]
Replacing  \((T, U) = (1, 1)\)  by  \((X, Y)\)  under the transformation  (2.5), we see that  \(X = -i\) and  \(Y = 0.\)  Hence we have 
\begin{equation} 
\tilde{f}_{n}(1, 1) = c_0(-i)^ {n+1}.
\end{equation}
Substituting this into (2.18), we have 
\begin{equation*} 
Dis( \tilde{g}_{n}) 
=  (-1)^{n(n+1)/2}\prod_{j < k }({t_j}-{t_k})^2\times 2^{n(n+1)}/i^{2n(n+1)} = 2^{n(n+1)}det H_n.
\end{equation*}
 Case 2 :  \quad 
\(n\) is odd.  The proof is similar.
\end{proof}
Since \( \tilde{g}_{n}(X, 1)\)  is a real polynomial, its discriminant   \(Dis( \tilde{g}_{n}(X, 1))\) is real.  Thus 
\(det \enskip H_n\)  is also real.  We may regard the space 
\[R_n : = \{(\zeta_1, \dots , \zeta_n) : \zeta_j = \overline{\zeta}_{n+1-j} , \enskip j = 1, \dots , n\}\]
as the space  \(\mathbb{R}^n\).  The set  \(\{det\enskip H_n = 0\} \)  may be seen as a hypersurface in  \(R_n\).
We define a subset   \(W_n\)  of  \(R_n\) by
 \[W_n = \{(\zeta_1, \dots , \zeta_n) : \zeta_j \enskip  \mbox { is the j-th elementary symmetric function in}\]
\[e^{i\theta _1}, \dots ,  e^{i\theta _{n+1}}, \enskip (j = 1, 2, \dots , n) \enskip \mbox{and} \enskip \theta_j \ne \theta_k \enskip \mbox{for}\enskip j \ne k\}.\]
Any element of  \(W_n\) corresponds to the \( \tilde{f}_{n}(T, 1)\)   all the roots of which are on the unit circle \(C\) and distinct.

We consider the n-simplex  \(S_n\) defined by 
\[S_n : = \{(\theta_1, \dots , \theta_n, \theta_{n+1}) :  \sum_{j=1}^{n+1}\theta_j = 0,\]
\[\theta_{n+1} \le  \theta_1 \le \dots   \le  \theta_n \le  2\pi + \theta_{n+1} \}.\]
We define a map \(\phi\)  from  \(S_n\)  to  \(R_n\)  by 
\[\phi(\theta_1 , \theta_2,  \dots, \theta_{n+1})   =  (\zeta_1, \dots , \zeta_n),  \]
where \enskip
\( \zeta_j \)\enskip  is the j-th elementary symmetric function in\enskip
\(e^{i\theta _1}, \dots ,  e^{i\theta _{n+1}}, \enskip\\ (j = 1, 2, \dots , n) .\)  \enskip By  \cite{EL}, we see that \(\phi\) is a diffeomorphism from int\((S_n)\)  onto  \(W_n\)  and \(\partial S_n\)  is mapped into the set \enskip  \(\{det\enskip H_n = 0\} \) .\\
Then the set  \(W_n\) is a connected component in  \(R_n \setminus \{det\enskip H_n = 0\}\)  that contains the origin  \enskip \( (\zeta_1=  \zeta_2 = \dots = \zeta_n = 0)\).

We consider the case that  \({\tilde f}_n(T, 1) \)   has exactly  \(k\) roots on the unit circle  \(C\)  .  Note that 
\(n+1-k\)  is even.
\begin{cor} \label{corollary:2.10}
We assume that  \({\tilde f}_n(T, 1) \)  has exactly  \(k\)  roots on \(C\)   and  that all the roots of   \({\tilde f}_n(T, 1) \)  are distinct.  Then
\[sgn \enskip det \enskip H_n = (-1)^{(n+1-k)/2},\]
where \enskip \(sgn \enskip det \enskip H_n \)  denotes the signature of \enskip \(det \enskip H_n .\)
\end{cor}
\begin{proof}\quad
The assumption implies that    \({\tilde g}_n(X, 1) \)   has  \((n+1-k)/2\)  pairs of  complex conjugate roots and \(k\) real  roots .  Then the corollary  follows from (2.16) and Theorem 2.9.
 \end{proof}

In theorem 2.9, we assume that  \(c_0 \ne 0\).  Here we consider this condition.   If  \(c_0 = 0,\) the positive integer \(k\) is uniquely determined by the condition that  \(c_k \ne 0\) and   \(c_j = 0\)   for  \(0 \le j < k.\)  We may consider  \( \tilde{g}_{n}(X, Y)/Y^k.\)   We study this case through the corresponding self- inversive forms.  Let \( \tilde{f}_{n}(T, U)\)  be the corresponding self-inversive form of\enskip  \(\tilde{g}_{n}(X, Y)\)  \enskip under  \(\Phi_{n}^{-1}\).  Then it is equivalent to consider   \( \tilde{f}_{n}(T, U)/(T-U)^k\).  We will see that 
\( \tilde{f}_{n}(T, U)/(T-U)^k \in {\mathcal {MA}}_{n-k}\)  by the following proposition.
 \begin{pro} \label{pro:2.11}
Assume that  \( \tilde{f}_{n}(T, U)\)  is an element of  \({\mathcal {MA}}_n \) and has a root \(T=U\).  Then  \( \tilde{f}_{n}(T, U)/(T-U)\)  is an element of  \({\mathcal {MA}}_{n-1} \).
\end{pro}
\begin{proof}\quad We may assume that 
\[\tilde{f}_{n}(T, U) = \prod_{j =1 }^{n+1}(T-{t_jU}) .\]
and  \( t_1t_2 \dots t_{n+1} = 1\).
It is known e.g. in \cite{M}  that a polynomial whose roots are on or symmetric in the unit circle \(C\) is a self-inversive polynomial and that the converse is true.  Hence all the roots of  \( \tilde{f}_{n}(T, U)\) are on or symmetric in \(C\).   Since  \( \tilde{f}_{n}(T, U)\)  has a root  \(T= U\),  we may assume \(t_1  =1\).  Then the roots of  \( \tilde{f}_{n}(T, U)/(T-U)\)  are on or symmetric in \(C\)  and satisfy  \( t_2t_3 \dots t_{n+1} = 1\). So   \( \tilde{f}_{n}(T, U)/(T-U)\) is self-inversive and the coefficients of \(T^n\) and \(U^n\)   are 1 and  \((-1)^n\), respectively.    Hence  \( \tilde{f}_{n}(T, U)/(T-U) \in {\mathcal {MA}}_{n-1}\).
\end{proof}
We note a remark.  In the proof of Lemma 2.8, we replace \(\zeta_0Z^{n+1}\)  by \((Z')^{n+1}\)  to make \(\zeta_0 = 1\).  If we use another replacement \((Z'')^{n+1}\) of  \(\zeta_0Z^{n+1}\),  then \enskip \(Z'' = \omega  Z'\)  where \enskip  \(\omega ^{n+1} = 1\).  The determinant \(det \enskip H_n\) does not change under this new replacement.
\subsection{Discriminants of monic self-reciprocal polynomials in \({\mathcal B}_n \)}\hspace{0cm}

In the section we assume that \(n\) is even.  We consider  monic self-reciprocal polynomials in \({\mathcal B}_n \) and their discriminants.  Let  \({\mathcal {MB}}_n \)  denote  self-reciprocal forms that are written in (2.10) satisfying  \(\zeta_0 = \zeta_{n+1} = 1.\)    The image of \({\mathcal {MB}}_n \) under 
 \(\Psi_{n}\) is denoted by \(\mathcal {MF'}_n \).  Any element of \({\mathcal {MF'}}_n \)  is written as
\begin{equation} 
\begin{split}
 \tilde{q}_{n}(X, Y) = d_0X^{n+1} + d_1X^{n}Y +  \cdot \cdot \cdot + d_nXY^{n} + d_{n+1}Y^{n+1}, \\
d_j \in {\mathbb R}, \quad j = 0, 1, \dots, n+1.
\end{split}
\end{equation}

To define the discriminant of \( \tilde{q}_{n}(X, 1)\) we assume that  \(d_0 \ne 0\).  We define the discriminant by

\begin{equation}
Dis(\tilde{q}_{n}(X, 1)) = d_0^{2n} \mathop{\Pi}  _{1 \le j < k \le n+1}(x'_j - x'_k)^2,
\end{equation}
where  \(x'_1, x'_2, \dots , x'_{n+1}\)  are roots of \(\tilde{ q}_{n}(X, 1)\).  

We consider monic self-reciprocal  polynomials in  \({\mathcal B}_n \).  We denote the image of  \( \tilde{q}_{n}(X, Y)\)   in  (2.20) under the map \(\Psi_n^{-1}\) by 
\begin{equation} 
 \tilde{p}_{n}(T, U) = T^{n+1} + \zeta_1T^{n}U + \zeta_2T^{n-1}U^2 + \cdot \cdot \cdot + \zeta_nTU^{n} + U^{n+1}.
\end{equation}
We consider the roots of  \( \tilde{p}_{n}\)  and \( \tilde{q}_{n}\).  We use the notation in (2.13).  In this case we may set   \(u'_j = y'_j = 1\)  for   \(j =1, \dots , n+1\).     Hence  \(x'_1, x'_2, \dots , x'_{n+1}\)  are roots of \( \tilde{q}_{n}(X, 1)\)  and   \(t'_1, t'_2, \dots, t'_{n+1} \) are roots of   \( \tilde{p}_{n}(T, 1)\).

  The coefficient  \(\zeta_j\) is  \((-1)^j\) times the \(j\)-th elementary symmetric function in \(t'_1,\dots, t'_{n+1}\). Clearly  \(t'_1,\dots, t'_{n+1} = -1\).   We define the \(m\)-th power \(k_m(m \in {\mathbb Z})\)  by
\[k_m(\zeta_1, \dots , \zeta_n) :  =  \sum_{j=1}^{n+1}(t'_j)^m.\]
Clearly
\[k_m((-1)\zeta_1,  (-1)^2\zeta_{2} , \dots , (-1)^n\zeta_n) = h_m(\zeta_1, \zeta_2, \dots , \zeta_n), \enskip \mbox{for} \enskip m \geq 0.\]

We define a  \((n+1)\times(n+1)\)  matrix  \(K_n\)  by
\[K_n :  = \left(
\begin{array}{cccc}
k_0 & k_{-1} &\dots & k_{-n}\\
k_1 & k_{0} &\dots & k_{-n+1}\\
\vdots &   &\ddots & \vdots\\
k_n & k_{n-1} &\dots & k_{0}
\end{array}
\right).\]
In our case
\[\zeta_j = \overline{\zeta}_{n+1-j} , \enskip j = 1, \dots , n.\]
\begin{theorem} \label{theorem:2.12}
Under the above notations we have 
\[Dis( \tilde{q}_{n}(X, 1)) =  2^{n(n+1)}det\enskip H_n.\] 
\end{theorem}
\begin{proof}
The proof is essentially the same as that of Theorem 2.9.  For any monic self-reciprocal polynomial \( \tilde{p}_{n}\) in (2.22) ,   we have \(\tilde{p}_{n}(1, 1) = d_0\),  by (2.11) .  This corresponds to \(\tilde{f}_{n}(1, 1) = c_0(-i)^ {n+1}\)  in (2.19).
Since  \(n\) is even,  it follows that \enskip  \((t'_1,  \dots,  t'_{n+1} )^n = 1\).

Then we have 
\[Dis( \tilde{q}_{n}) =  2^{n(n+1)}det\enskip K_n.\] 
Clearly

\[- \tilde{p}_{n}(-T, U) = \tilde{f}_{n}(T, U).\]
Hence by (2.8) and (2.13) we have with a suitable numbering  
\[t'_j = -t_j, \enskip j = 1, 2, \dots, n+1.\]
Then  \[(t'_j - t'_k)^2  =  ( t_{j}-t_k )^2.\]

Therefore by the proof of Theorem 2.9 we have
\[det \enskip H_n = det \enskip K_n.\]
\end{proof}


\bibliographystyle{amsplain}

\end{document}